\documentclass[11pt, reqno]{amsart}

\usepackage[margin=3.6cm]{geometry}
\usepackage{amsmath, amsfonts, amsthm, amssymb, dsfont, float}
\usepackage{exscale}
\usepackage{epsfig}
\usepackage{amscd}
\usepackage{cite}
\usepackage{graphicx}
\usepackage{indentfirst} 

\renewcommand{\geq}{\geqslant}
\renewcommand{\leq}{\leqslant}

\def\p{\partial}

\def\basic_energy{
\tilde{E}(t) = \int_{\Omega} | \partial_tu |^2 + | \nabla u |^2 dV 
}

\def\modified_energy{
 | \partial_tu_1 |^2+| \nabla u_0 |^2 
}

\def\domain{
\int_{0}^{T}\int_{\Omega'}
}

\def\1st_energy_term{
\int_{\Omega} \partial_t^2 u u_t
}
\def\initial_formula{
\int_0^T \int_{\Omega'} [\partial^2_t + P, Y]u u dV dt
}
\def\boundary_term1{
\int_{\Omega'} \partial_t u u dV\Big|_0^T
}
\def\side_term{
\int_0^T \int_{\partial \Omega'} Yu (\nu^T BB^T) \nabla u dS dt\
}

\def\simp_side_term{
\int_0^T \int_{F_0'} \frac{1}{\sqrt{n}} \partial_{\nu} u(\nu^T BB^T)\nabla u dS_0'dt
}

\def\inside_side_term{
\partial_{\nu} u(\nu^T BB^T)\nabla u 
}
\def\explicit_vector_field{
y_1 \partial_{y_1} + y_2 \partial_{y_2}  + \dots + y_n \partial_{y_n}
}
\def\Yu_case{
Yu\Big|_
}
\def\vertical_normal_vector{
\frac{1}{\sqrt{n}}
\begin{bmatrix}
     1\\ \vdots \\ 1
\end{bmatrix}
}

\def\tan_vector{
\frac{1}{\sqrt{n}}
\begin{bmatrix}
1\\ -1\\0\\ \vdots\\0\\
\end{bmatrix}
}

\theoremstyle{definition}
\newtheorem{definition}{Definition}[section]
\newtheorem{theorem}{Theorem}[section]
\newtheorem{lemma}[theorem]{Lemma}
\numberwithin{equation}{section}
\title{ASYMPTOTIC BOUNDARY OBSERVABILITY FOR THE WAVE EQUATION ON SIMPLICES}
\author{Hans Christianson and Ziqing Lu}
\begin{document}

\begin{abstract}
In this paper, we consider the wave equation on an n-dimensional
simplex with Dirichlet boundary conditions. Our main result is an
asymptotic observability identity from any one face of the
simplex. The novel aspects of the result are that it is a large-time
asymptotic rather than an estimate, and it requires no dynamical
assumptions on the billiard flow. The proof is an adaption of the
techniques from \cite{Chr17,Chr18,CS}, using mainly integrations by parts.
\end{abstract}

\maketitle

\section{Introduction}
In this paper, we study the wave equation $(\partial_t^2 - \Delta) u =
0 $ on an  n-dimensional simplex with Dirichlet boundary conditions. We obtain an asymptotic observability property from any one face of the simplex. This generalizes the result in \cite{CS} from triangles  to simplices in higher dimensions. The proof is similar to that of \cite{CS}. It mainly uses commutators and integration by parts arguments, but involves a coordinate transformation and linear algebra as well. 

The formal statement of the problem is represented by \eqref{intro:1}:
\begin{equation} \label{intro:1}
\begin{cases}
  (\partial_t^2 - \Delta) u = 0 \text{ on } (0,\infty)\times\Omega, \\
  u\Bigr|_{\substack{\partial \Omega}} = 0,\\
  u(0,x_1, x_2, \dots, x_n) = u_0(x_1, x_2, \dots, x_n), \\
  u_t(0,x_1, x_2, \dots, x_n) = u_1(x_1, x_2, \dots, x_n)\\
\end{cases}
\end{equation}
where  $u$ is real-valued and $u_0 \in H_0^1(\Omega) \cap H^3(\Omega)$ and $u_1 \in H_0^1(\Omega) \cap H^2(\Omega)$. Regarding this problem, the main theorem is the following:

\begin{theorem}
\textit{Let $\Omega \subseteq \mathbb{R}^n$ be a simplex with faces $F_0, F_1, F_2, \dots, F_n$ and suppose $u$ solves the wave equation \eqref{intro:1} on $\Omega$. For any finite time $T > 0$, we obtain the following asymptotic observability identity for any one face $F_j$, $0 \leq j \leq n$ of the simplex $\Omega$:
\begin{equation}
\int_0^T \int_{F_j} |\partial_{\nu} u|^2 dS_j dt = \frac{TArea(F_j)}{nVol(\Omega)}\tilde{E}(0)\left(1+\mathcal{O}\left( \frac{1}{T}\right)\right) ,
\end{equation}
where $\partial_{\nu}u$ is the normal derivative on $F_j$ and $dS_j$ is the induced surface measure. Here $\tilde{E}(t)$ is the conserved energy of the solution $u$ to the wave equation, defined by:
\begin{equation}
\basic_energy. \label{energy_initial}
\end{equation}
Remark 1.Observability in this paper means we can observe the initial energy by taking a measurement on one face. }
\end{theorem}

\section{History}
The study of observability is based on the prerequisite that waves propagate along straight-line paths in a homogeneous medium. Waves reflect off the boundary obeying the law of reflection, so that the angle of incidence is the same as the reflection angle. 

The idea of observability (or, more precisely, closely related
geometric control) originates from the paper of Rauch-Taylor paper \cite{RT}, where they studied geometric control for the damped wave equation $u_{tt} - \Delta u + a(x) \partial_t u = 0$. The idea is if every ray passes through the damping region where $a >0$, the energy decays exponentially as $E(t) \leq Ce^{-\frac{t}{c}}E(0)$.  For example, the first picture of Figure \ref{fig:01} is not a geometric control while the second one is.

\begin{figure}[H]
\centering
\includegraphics[scale=0.35]{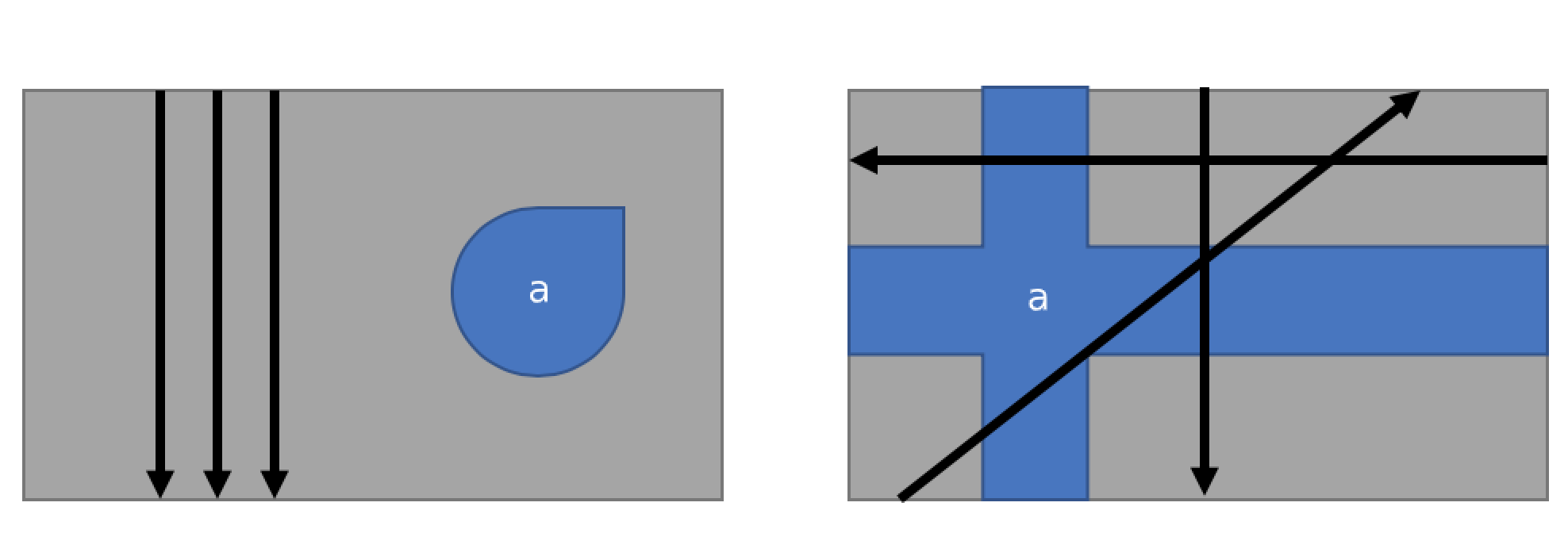}
\caption{Rays Passing Through Subsets Of Domain}
\label{fig:01}
\end{figure}

The closely related idea of observability asks if you can ``see" the initial energy by taking the measurement of a subset of the domain or a subset of the boundary. 
In the work of Bardos-Lebeau-Rauch \cite{BLR}, the observability from a subset of the boundary was studied in depth. The condition for observability, similar to in Rauch-Taylor is that all rays hit the observability region on the boundary transversally, as indicated in Figure \ref{fig:02}.
\begin{figure}[H]
\centering
\includegraphics[scale=0.35]{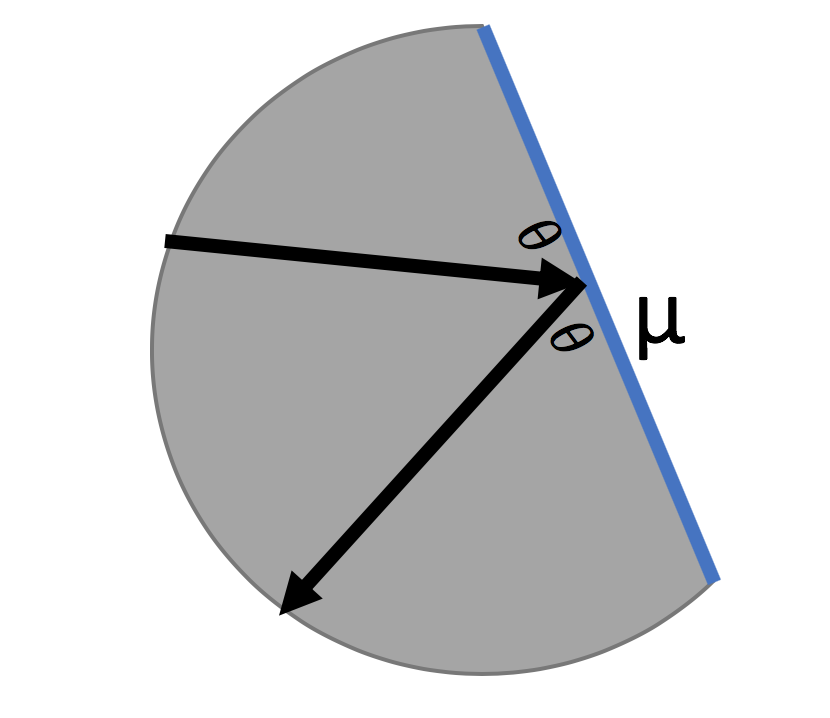}
\caption{Rays Hitting The Control Region On The Boundary}
\label{fig:02}
\end{figure}

In the work of Christianson-Stafford\cite{CS}, an asymptotic observability property from any one side of a triangle is proved. The proof was split into the cases of acute triangles and obtuse triangles, shown by Figure \ref{fig:03}. Waves are assumed to propagate along straight-line paths at unit speed, traveling from the opposite corner to the interested side. The result was obtained with an argument of the method similar to the present paper, by the use of commutator and integration by parts arguments.
\begin{figure}
\centering
\includegraphics[scale=0.25]{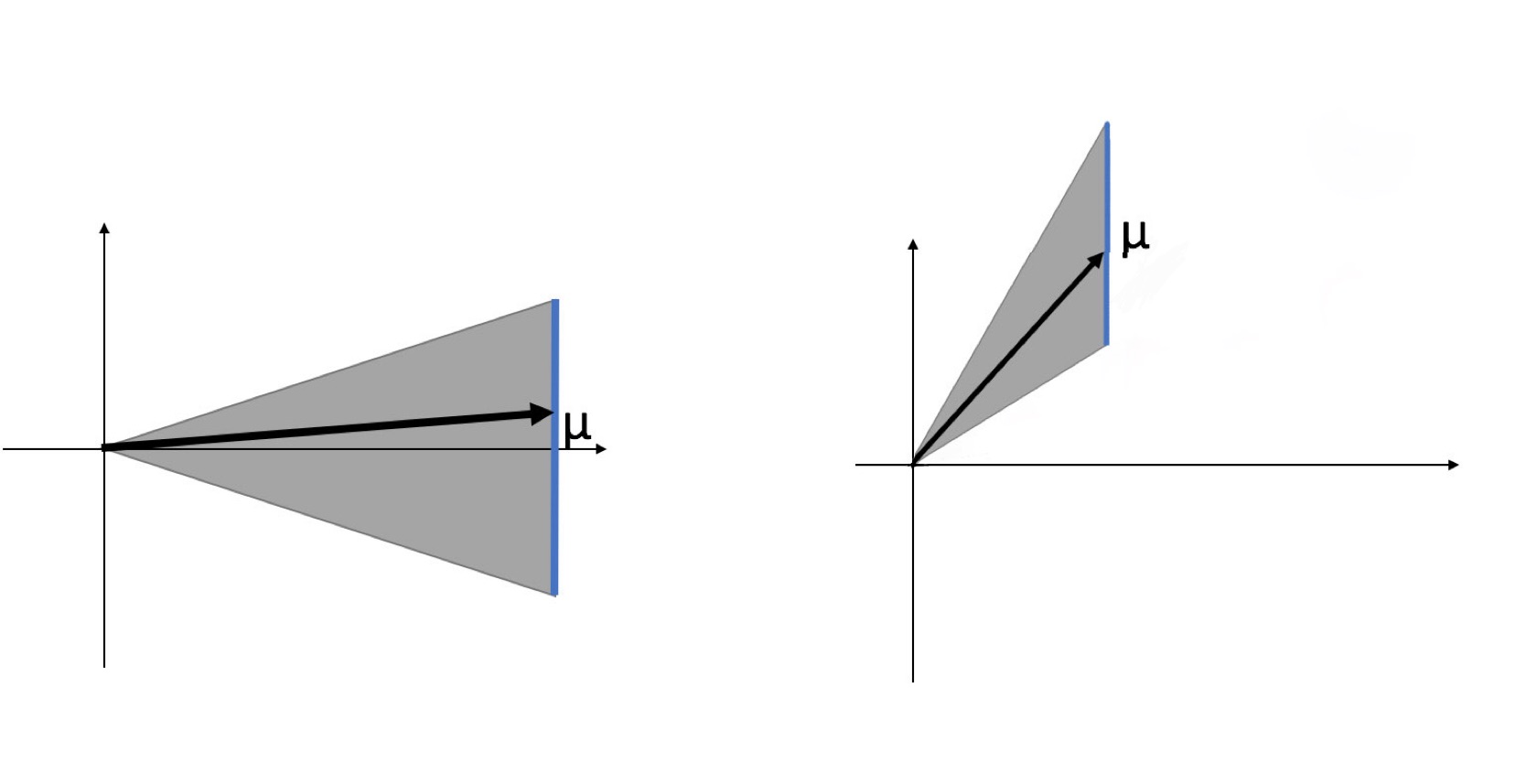}
\caption{Asymptotic Observability On One Side Of Triangles}
\label{fig:03}
\end{figure}
\section{Preliminaries}
This section of preliminaries provides lemmas and definitions required in the main proof.
\begin{lemma}[Conserved Energy] \textit{For the solution $u$ to the wave equation \eqref{intro:1}. The energy is conserved: \label{Lemma:1}
\begin{equation}
\tilde{E}(t) = \tilde{E}(0) .
\end{equation}}
\end{lemma}
\theoremstyle{definition}
\begin{definition}[Elliptic Operator]
A constant coefficient elliptic operator  $P$ on a domain $\Omega$ in $ \mathbb{R}^n$ is defined by
\begin{equation}
P = -\sum_{i,j=0}^n K_{ij}\partial_{x_i}\partial_{x_j},\label{elliptic-op}
\end{equation}
where $K$ is an $n \times n$ symmetric, positive definite matrix.
\end{definition}

\begin{lemma}[Ellipticity]\label{thm:2.2} \textit{Let $\Omega \subseteq \mathbb{R}^n$ be a simplex. If $u \in H_0^1(\Omega) \cap H^3(\Omega)$, then there exists an constant $C  \text{ such that}$
\begin{equation}
 \| \nabla u \|^2_{L^2(\Omega)} \leq C \langle Pu, u\rangle_{L^2}.
\end{equation}}
\end{lemma}
The following lemma is a modified version of Theorem 4 published in
the paper \cite{Chr18} by the first author. The modified version can be directly used in the main proof of this paper.
\begin{lemma}[Green's formula for \eqref{elliptic-op}] \label{thm:2.3} \textit{Let $\Omega' \in \mathbb{R}^n$ be the standard simplex and $g\Big|_{\partial \Omega'}=0$. Let $P$ be an elliptic operator. Then for functions $f, g \in C^{\infty} (\Omega)$, we have,
\begin{equation}
\int_{\Omega'} (Pf)g dV = \int_{\Omega'} f(Pg)dV + \int_{\partial\Omega'} f(\nu^TK)\partial gdSdt
\end{equation}
where $\nu$ is the outward unit vector on every face of the simplex $\Omega$.}
\end{lemma}

This paper inherits the notation that Christianson used in the paper \cite{Chr18} to define higher dimension simplicies. 
\theoremstyle{definition}
\begin{definition}[Simplex]\label{def:1}
Let independent vectors $\vec{p_1}, \vec{p_2}, \dots, \vec{p_n}\in \mathbb{R}^n$ span from the origin, then a simplex $\Omega$ in $\mathbb{R}^n$ is defined as:
\begin{equation}
\Omega =\{ \sum_{i=1}^n c_i \vec{p_i}: \sum_{i=0}^n c_i \leq 1 \ and \  c_i \geq 0, \vec{p_i} \in \mathbb{R}^n \}
\end{equation}
\end{definition}

We denote the face where $c_i = 0, i = 1, \dots, n$ as $F_i$ and the remaining face $F_0$\footnote{By choosing a different corner as the origin, we can get the result for any one of the faces.}. Let matrix A be $
\begin{bmatrix}
| & | & \ & | \\
\vec{p_1} & \vec{p_2} & \dots &\vec{ p_n} \\ 
| & | & \ & | 
\end{bmatrix}
$. Because the column vectors $\vec{p_1}, \vec{p_2}, \dots, \vec{p_n}$ are linearly independent, there exists an inverse matrix of $A$. Denote this inverse matrix by $B$.

In particular, we have standard simplex $\Omega' \in \mathbb{R}^n$.

\theoremstyle{definition}
\begin{definition}[Standard Simplex]\label{def:2}
Let unit vectors $\vec{e_1 }= [1, 0, 0, \dots, 0]^T,\\ \vec{e_2}
=[0,1,0,\dots, 0]^T, \dots, \vec{e_n} = [0,0,0,\dots, 1]^T \in
\mathbb{R}^n$ be the $n$ standard basis of linearly independent orthonormal vectors.  The standard simplex, denoted by $\Omega'$, is defined by all convex combinations of these linearly independent unit vectors:
\begin{equation}
\Omega' =\{ \sum_{i=1}^n d_i \vec{e_i}: \sum_{i=0}^n d_i \leq 1 \ and \  d_i \geq 0, \vec{e_i} \in \mathbb{R}^n \}
\end{equation}
\end{definition}

This standard simplex has $n+1$ faces $F'_0, F'_1, \dots, F'_n$, where $F_i'$ is the face with $d_i = 0, i = 1, 2, \dots, n$ while the remaining face  is $F'_0$. 

The standard rectangular coordinates of the standard simplex $\Omega'$ are denoted as $(y_1, y_2, \dots, y_n)$ in $\mathbb{R}^n$ while the rectangular coordinates of the original simplex $\Omega$ are denoted as $(x_1, x_2, \dots, x_n)$ in this paper.

The following transformation takes the arbitrary simplex $\Omega$ to the standard simplex $\Omega'$.  Using Definition \eqref{def:1} and Definition \eqref{def:2}, let $\vec{y} = [y_1, y_2, \dots, y_n]$ be one vector in the standard simplex $\Omega'$. Then by the following transformation:
\begin{equation}
\vec{x} = A\vec{y},
\end{equation}
we can obtain the corresponding vector $\vec{x}= [x_1, x_2, \dots, x_n]$ in the simplex $\Omega$. By this transformation, we changed the standard basis $\{\vec{e_j}: j=1,\dots,n \}$ into the set of basis $\{ \vec{p_j}: j=1,\dots,n \}$.


By observing the relation $\vec{x} = A\vec{y}$,  we claim that $\nabla_x = B^T \nabla_y$ after the transformation. The proof of this claim is in Appendix \ref{appendix: change variables}.

Since the Laplacian operator is $-\Delta_x = -\nabla^T_x \nabla_x$  on the simplex $\Omega$,  the above claim implies this equation is equivalent to $P  = -(B^T\nabla_y)^T(B^T\nabla_y)=-\nabla^T_y BB^T \nabla_y$ on the standard simplex $\Omega'$. Denote it by $P$.

From now on, $\nabla$ in this paper means $\nabla_y$ by default.

According to \eqref{energy_initial}, the energy of the solution to the wave equation in $y$-coordinates can be defined as :
\begin{equation}
E(t) = \int_{\Omega'} | u_t |^2 + | B^T\nabla u |^2 dV.
\end{equation}
We claim that this energy is also conserved: 
\begin{equation}
E(t) = E(0). \label{conserved-energy}
\end{equation}
\begin{proof}[Proof of (\ref{conserved-energy})]
\begin{equation}
\begin{split}
E(t)' & = \int_{\Omega'} u_{tt}u_t + u_t u_{tt} + (B^T \nabla u_t)(B^T \nabla u) + (B^T \nabla u)(B^T \nabla u_t)dV\\
& = 2\int_{\Omega'} u_{tt} u_t + (B^T \nabla u)(B^T \nabla u_t)dV\\
& = 2\int_{\Omega'} (u_{tt} -\nabla^T B B^T \nabla u)u_t dV\\
& = 0.
\end{split}
\end{equation}
\end{proof}
 
\begin{lemma}\label{Lemma:2}
Consider the vector field $X = \sum_{i=0}^n x_i\partial_{x_i}$ and the second order constant coefficient symmetric operator $P = -\sum_{i,j=1}^n a_{ij}\partial_{x_i} \partial_{x_j}$. Then:
\begin{equation}
[P,X] = 2P
\end{equation}
 \end{lemma}
\begin{proof}
When $i=j$, $i = 1, \dots, n$,
\begin{equation}
\begin{split}
&[-a_{ii}\partial_{x_i}^2, x_1\partial_{x_1} + x_2\partial_{x_2} + \dots + x_n \partial_{x_n}]\\
& = \sum_{k=1}^n [-a_{ii}\partial_{x_i}^2, x_k\partial_{x_k}]\\
& = [-a_{ii}\partial_{x_i}^2, x_i\partial_{x_i}] + \sum_{k=1, k \neq i}^n [-a_{ii}\partial_{x_i}^2, x_k\partial_{x_k}]\\
& = -a_{ii}\partial_{x_i}^2(x_i\partial_{x_i}) + x_i\partial_{x_i}(a_{ii}\partial_{x_i}^2)
+ \sum_{k=1, k \neq i}^n( -a_{ii}\partial_{x_i}^2(x_k\partial_{x_k}) + x_k\partial_{x_k}(a_{ii}\partial_{x_i}^2))\\
& = -a_{ii}\partial_{x_i}(\partial_{x_i} + x_i\partial^2_{x_i}) + x_ia_{ii}\partial_{x_i}^3 + 0\\
& = -a_{ii}\partial^2_{x_i} - a_{ii}\partial_{x_i}^2 - x_ia_{ii}\partial^3_{x_i}+ x_ia_{ii}\partial_{x_i}^3\\
& = -2a_{ii}\partial_{x_i}^2\\
\end{split}
\end{equation}
When $i \neq j$, $i,j = 1, \dots, n$,
\begin{equation}
\begin{split}
&[-a_{ij}\partial_{x_i} \partial_{x_j}, x_1\partial_{x_1} + x_2\partial_{x_2} + \dots + x_n\partial_{x_n}]\\
& = [-a_{ij}\partial_{x_i} \partial_{x_j}, x_i\partial_{x_i}] + [-a_{ij}\partial_{x_i} \partial_{x_j}, x_j\partial_{x_j}] + \sum_{k = 1, k \neq i,j}^n [-a_{ij}\partial_{x_i} \partial_{x_j}, x_k\partial_{x_k}]\\
& = [-a_{ij}\partial_{x_i} \partial_{x_j}, x_i\partial_{x_i}] + [-a_{ij}\partial_{x_i} \partial_{x_j}, x_j\partial_{x_j}] + 0\\
& = -a_{ij}\partial_{x_i} \partial_{x_j}(x_i\partial_{x_i}) + x_i\partial_{x_i}(a_{ij}\partial_{x_i} \partial_{x_j}) \\
& \quad - a_{ij}\partial_{x_i} \partial_{x_j}(x_j\partial_{x_j}) + x_j\partial_{x_j}(a_{ij}\partial_{x_i} \partial_{x_j})\\
& = -a_{ij}\partial_{x_j}(\partial_{x_i} + x_i\partial^2_{x_i}) + x_i a_{ij}\partial^2_{x_i}\partial_{x_j}\\
& \quad - a_{ij}\partial_{x_i}(\partial_{x_j} + x_j\partial^2_{x_j}) + x_j a_{ij}\partial_{x_j}^2\partial_{x_i}\\
& = -2a_{ij}\partial_{x_j} \partial_{x_i}\\
\end{split}
\end{equation}
\end{proof}
 
\section{Proof of the Theorem}
Let the vector field be $Y = y_1\partial_{y_1} + y_2\partial_{y_2} + \dots + y_n\partial_{y_n}$ on the standard simplex $\Omega'$. For $P = - \nabla^T B B^T \nabla$, we will compute: 
\begin{equation}
\int_0^T \int_{\Omega'} [\partial_t^2 + P, Y]u u dV, \label{wezy}
\end{equation}
in two different ways. The two approaches are based on these two different ways of dealing with the commutator. One approach starts with using Lemma \ref{Lemma:2} while in the other approach we evaluate the commutator explicitly.

To start our first approach, we could apply Lemma  \ref{Lemma:2} to compute the commutator, because $P$ is an elliptic operator. Then, since $u$ the satisfies wave equation, integration by parts gives: 
\begin{equation}\label{eqn:6}
\begin{split}
\initial_formula &= \domain 2Pu udVdt \\
& = \domain (P - \partial^2_t) u udVdt \\
& = \domain -\nabla^T BB^T \nabla u u dVdt - \domain \partial_t^2 u u dVdt \\
& = \domain -(B^T \nabla)^T B^T \nabla u u dVdt -  \domain \partial_t^2 u u dVdt \\
& = \domain (B^T \nabla) u  (B^T \nabla)u dV dt + \domain \partial_t u \partial_t u dVdt \\
&\quad - \int_{\Omega'} \partial_t u u dV\Big|_{0}^T \\
& = \domain |  B^T \nabla u |^2 dVdt + \int_{0}^{T}\int_{\Omega'} | \partial_t u|^2 dVdt \\
& \quad - \boundary_term1\\
& = \domain |  B^T \nabla u |^2 + | \partial_t u |^2 dVdt  - \boundary_term1 \\
& = TE(0) - \boundary_term1\\
\end{split}
\end{equation}
Notice that at the last step of the previous computation, we used result of \eqref{conserved-energy}. 

We next compute \eqref{wezy} by a different approach. We first evaluate the commutator explicitly and then use integration by parts. The second term generated by the commutator cancels out because of the homogeneous wave equation. Indeed,
\begin{equation}
\begin{split}
\initial_formula &=\domain (\partial_t^2 + P)Yu u - Y (\partial^2_t + P) u u dVdt\\
& = \domain \partial_t^2 Y uu + PYu u dVdt
\end{split}
\end{equation}
After we apply Lemma \ref{thm:2.3} to the second term and use integration by parts twice on the first term, we have:
\begin{equation}\label{result:1}
\begin{split}
 & \quad \ \domain \partial_t^2 Y u u + PYu u dVdt\\
 & = \domain Yu \partial^2_t u dVdt + \int_{\Omega'} \partial_tYu udV\Big|_0^T- \int_{\Omega'} Yu \partial_t u dV \Big|_0^T \\
 & \quad + \domain YuPudVdt 
+ \side_term\\
& = \domain Yu(\partial^2_t + P)u dVdt + \int_{\Omega'} \partial_tYu udV\Big|_0^T  \\
& \quad - \int_{\Omega'} Yu \partial_t u dV \Big|_0^T + \side_term\\
& = \int_{\Omega'} \partial_tYu udV\Big|_0^T  - \int_{\Omega'} Yu \partial_t u dV \Big|_0^T + \side_term
\end{split}
\end{equation}
 where $\nu$ is the outward normal vector to every face, and $dS$ is the reduced differential displacement. 

To simplify the term integrate on the boundary in \eqref{result:1}, we study every face of the simplex by writing out the vector field $Y = \explicit_vector_field$. On face $F_1'$, we have that: 
\begin{equation}
\begin{split}
\Yu_case{F_1'} & = (\explicit_vector_field)u \\
& = (0\cdot \partial_{y_1})u + y_2 \cdot 0 + y_3 \cdot 0 + \dots + y_n \cdot 0 \\
& = 0
\end{split}
\end{equation}
by observing that $y_1 = 0$ on $F_1'$ and that the tangential derivatives $\partial_{y_2}u, \partial_{y_3}u, \dots, \partial_{y_n}u$ of $F_1'$ are all equal to 0 since $u\Big|_{\partial\Omega'} = 0$. Therefore, we could conclude that $Yu\Big|_{F_1'} = 0$.
Similarly, the same result applies on the other $n-1$ faces of the standard simplex:
\begin{equation}
\begin{cases}
\Yu_case{F_1'} = (\explicit_vector_field)u = 0\\
\Yu_case{F_2'} = (\explicit_vector_field)u = 0\\
\vdots\\
\Yu_case{F_n'} = (\explicit_vector_field)u = 0.
\end{cases}
\end{equation}

However, on face $F_0'$, the condition is different because none of the spatial variables is 0 or none of $\partial_{y_2}u, \partial_{y_3}u \dots, \partial_{y_n}u$ are tangential derivatives. We need to find its tangential vectors. 

Notice that the unit normal derivative on this face is $\partial_{\nu} u=\frac{1}{\sqrt{n}}[1,1,\dots,1]^T\nabla u$. Because the tangential derivatives are orthogonal to the normal derivative, we could choose different tangential vectors and take advantage of the fact that they all equal to 0. The first tangential derivative we choose is $\frac{1}{\sqrt{n}}[1, -1, 0, \dots, 0]^T\nabla u $ and it satisfies:
\begin{equation}
\frac{1}{\sqrt{n}}[1,-1,0,\dots,0]^T \nabla u  = 0 \quad \Rightarrow \quad \partial_{y_1}u = \partial_{y_2}u\\
\end{equation}
Similarly, by choosing other tangential derivatives for the face $F_0'$ and by setting them equal to 0, we conclude that:
\begin{equation}
\partial_{y_1}u  = \partial_{y_2}u  = \dots = \partial_{y_n}u
\end{equation}
Therefore, using the conclusion above, the normal vector can be represented as:
\begin{equation}
\begin{split}
\partial_{\nu} u & = \frac{1}{\sqrt{n}}(\partial_{y_1} + \partial_{y_2} + \dots + \partial _{y_n})u \\
& = \frac{1}{\sqrt{n}} ( n \partial_{y_1} u ) \\
& = \sqrt{n} \partial_{y_1}u \\
& = \sqrt{n} \partial_{y_2}u \\
& = \dots \\
& = \sqrt{n} \partial_{y_n}u\\
\end{split}
\end{equation}
which implies that:

\begin{equation}
\begin{split}
& \begin{cases}
\partial_{y_1}u = \frac{1}{\sqrt{n}} \partial_{\nu}u\\
\partial_{y_2} u= \frac{1}{\sqrt{n}} \partial_{\nu}u\\
\vdots\\
\partial_{y_n}u = \frac{1}{\sqrt{n}} \partial_{\nu}u\\
\end{cases}
\end{split}
\end{equation}
Since $y_1 + y_2 +\dots+y_n = 1$ on $F_0'$, 
\begin{equation}
\begin{split}
& \quad \ Yu\Big|_{F_0'} \\
& = (\explicit_vector_field) u \\
& = (y_1 + y_2 + \dots + y_n) \frac{1}{\sqrt{n}} \partial_{\nu}u \\
& = 1 \times \frac{1}{\sqrt{n}} \partial_{\nu} u\\
& =  \frac{1}{\sqrt{n}} \partial_{\nu} u .
\end{split}
\end{equation}
Now the integration on the boundary can be simplified into the form: 
\begin{align}
\simp_side_term,
\end{align}
which only involves the face $F_0'$. As a result, the second approach \eqref{result:1} is simplified to: 
\begin{equation} \label{result:2}
\begin{split}
\initial_formula & = \int_{\Omega'} \partial_tYu udV\Big|_0^T - \int_{\Omega'} Yu \partial t u dV \Big|_0^T\\
&\quad  + \simp_side_term
\end{split}
\end{equation}
We now study the first term of this simplified version \eqref{result:2}, using the integration by parts and the chain rule: 
\begin{equation} \label{result:3}
\begin{split}
\int_{\Omega'} \partial_tYu udV\Big|_0^T & = -\int_{\Omega'} \partial_t u \sum_{j=1}^n \partial_{y_j} (y_j u) dV \Big|_0^T \\
& = -\int_{\Omega'} \partial_{t} u (n u + \sum_{j=1}^n y_j \partial_{y_j} u ) dV\Big|_0^T\\
& = -n\int_{\Omega'} \partial_t u u dV\Big|_0^T- \int_{\Omega'} \partial_t u YudV\Big|_0^T.
\end{split}
\end{equation}
Then we have the second approach summarized as:
\begin{equation} \label{result:4}
\begin{split}
\initial_formula  & = -n\int_{\Omega'} \partial_t u u dV\Big|_0^T - 2\int_{\Omega'} \partial_t u YudV\Big|_0^T. \\
& \quad+ \simp_side_term
\end{split}
\end{equation}

Combining this and \eqref{eqn:6}, and re-organizing terms, we have:
\begin{equation} \label{result:5}
\simp_side_term = TE(0) + (n-1)\int_{\Omega'} \partial_t u u dV\Big|_0^T + 2\int_{\Omega'} \partial_t u YudV\Big|_0^T 
\end{equation}

Now to obtain the observability from face $F_0'$, we are going to analyze the last two terms of \eqref{result:5} to determine whether we could absorb them into initial energy through estimation. 

Firstly, to estimate the third term on the right side of \eqref{result:5}, for some fixed time $t_0$, we use Cauchy's inequality and triangle equality to obtain:
\begin{equation} \label{esti-I-I}
\begin{split}
\Big| \int_{\Omega'} \partial_t u YudV\Big|_{t_0}\Big| & \leq C\int_{\Omega'} | \partial_t u |^2dV \Big|_{t_0}  + C\int_{\Omega'}(\sum_{j=1}^n | y_j \partial_{y_j} u |)^2  dV \Big|_{t_0}\\
& \leq C\int_{\Omega'} | \partial_t u |^2dV \Big|_{t_0} + C \int_{\Omega'}(\sum_{j=1}^n | \partial_{y_j} u |)^2 dV \Big|_{t_0}\\
& \leq C\int_{\Omega'} | \partial_t u |^2dV \Big|_{t_0}  + C \int_{\Omega'} (\sum_{i=1} ^n | \partial_{y_i}u |)( \sum_{k=1}^n | \partial_{y_k}u |) dV \Big|_{t_0}\\
& = C\int_{\Omega'} | \partial_t u |^2dV \Big|_{t_0} + C \int_{\Omega'} \sum_{i=1} ^n \sum_{k=1}^n( | \partial_{y_i}u || \partial_{y_k}u | )dV\Big|_{t_0}\\
& \leq C\int_{\Omega'} | \partial_t u |^2dV \Big|_{t_0} + C\int_{\Omega'}(| \partial_{y_1}u |^2 + | \partial_{y_2}u |^2 + \dots + | \partial_{y_n}u |^2 ) dV \Big|_{t_0}\\
& \leq C \int_{\Omega'} | \partial_t u |^2+ | \nabla u |^2 dV \Big|_{t_0}\\
\end{split}
\end{equation}
Note that every coefficient $C$ changes from line to line but they are independent from time variable $t_0$. By applying the Lemma \ref{thm:2.2}, which can be done since $B^TB$ is positive definite, we have:
\begin{equation}\label{est-I-II}
\begin{split}
\int_{\Omega'} | \partial_t u |^2+ | \nabla u |^2 dV \Big|_{t_0} & \leq C\int_{\Omega'} | \partial_t u |^2+ \langle BB^T\nabla u, \nabla u \rangle dV \Big|_{t_0}\\
& \leq C\int_{\Omega'}  | \partial_t u |^2 + | B^T \nabla u |^2 dV \Big|_{t_0} \\
& = CE(0)
\end{split}
\end{equation}
Thus, combining \eqref{esti-I-I} and \eqref{est-I-II} gives :
\begin{equation} \label{est-I-III}
\begin{split}
\Big| \int_{\Omega'} \partial_t u YudV\Big|_0^T  \Big|& \leq \Big| \int_{\Omega'} \partial_t u YudV\Big|_{t = T} \Big| + \Big| \int_{\Omega'} \partial_t u YudV\Big|_{t = 0} \Big| \\
 & \leq CE(0).\\
 \end{split}
\end{equation}

Similarly, we perform another estimation for the second term of
\eqref{result:5} by using the Cauchy inequality and the Poincaré
inequality. Again, the coefficient $C$ changes but does not depend on
the time variable $t$.  Fixing a time $t = t_0$, we have 
\begin{equation}\label{est-II-I}
\begin{split}
 \left| \int_{\Omega'} \partial_t u u dV\Big|_{t_0} \right|  & \leq C\int_{\Omega'} | \partial_t u |^2dV \Big|_{t_0}  + C\int_{\Omega'}| u |^2 dV \Big|_{t_0}\\
& \leq (C\int_{\Omega'} | \partial_t u |^2dV \Big|_{t_0}  + C\int_{\Omega'} | \nabla u |^2 dV \Big|_{t_0}) \\
& \leq C\int_{\Omega'}  | \partial_t u |^2 +   \langle BB^T\nabla u, \nabla u \rangle dV \Big|_{t_0} \\
& \leq C\int_{\Omega'}  | \partial_t u |^2 + | B^T \nabla u |^2 dV \Big|_{t_0}\\
& = CE(0)
\end{split}
\end{equation}
by conservation of energy.  Applying this with $t_0 = 0$ and $t_0 =
T$, we have
\[
\left| \int_{\Omega'} \p_t u u dV\Big|_0^T \right| \leq C E(0).
\]
Therefore combining \eqref{est-II-I} and \eqref{est-I-III} into \eqref{result:5} yields:
\begin{equation} \label{con: 1}
\simp_side_term = TE(0) + \mathcal{O}(1)E(0)
\end{equation}

 To obtain the observability on face of the original simplex $\Omega$, we make the following transformation from the standard simplex $\Omega'$ back to the original simplex $\Omega$. 
 
Starting from the right side of \eqref{con: 1}, we first transform the energy back to the original simplex, by using the results introduced in \ref{V:1} and \ref{V:2} in Appendix \ref{appendix: Determinant and Volume}, $dy = \frac{1}{\det(A)} dx$ and $\det(A) = n!Vol(\Omega)$. Therefore,
 \begin{equation} \label{con:2}
 \begin{split}
TE(0) + \mathcal{O}(1)E(0)& = (T +\mathcal{O}(1))\int_{\Omega'} (| \partial_t u_1 |^2 + | B^T \nabla u_0 |^2) dV \\
 & = (T +\mathcal{O}(1)) \int_{\Omega'} (| \partial_t u_1 |^2 + | B^T \nabla u_0 |^2)dy_1 dy_2 \dots dy_n\\
 & = \frac{T +\mathcal{O}(1)}{det(A)} \int_{\Omega} \modified_energy dx_1 dx_2 \dots dx_n\\
 & = \frac{T +\mathcal{O}(1)}{n!Vol(\Omega)} \int_{\Omega} \modified_energy dx_1 dx_2 \dots dx_n \\
 & = \frac{T +\mathcal{O}(1)}{n!Vol(\Omega)} \tilde{E}(0)\\
 \end{split}
 \end{equation}
On the left side of \eqref{con: 1}, to transform from face $F_0'$ of standard simplex $\Omega'$ back to the face $F_0$ of original simplex $\Omega$, we first change the graph coordinate $dS_0'$ back to the rectangular coordinate:
\begin{equation}
\begin{split}
 F_0' & = \{y_n = 1 - y_1 - y_2 - \dots - y_{n-1}\} \\
\Rightarrow dS_0' & = (1^2 + (-1)^2 + \dots + (-1)^2)^{\frac{1}{2}} dy_1 dy_2 \dots dy_{n-1} \\
& = \sqrt{n} dy_1 dy_2 \dots dy_{n-1}. \\
\end{split}
\end{equation}
Then,  the left side of \eqref{con: 1} can be written as:
\begin{equation} \label{con:3}
\begin{split}
\simp_side_term & = \int_0^T \int_{\Omega'_{n-1}}  \frac{\sqrt{n}}{\sqrt{n}} \inside_side_term  dy_1 dy_2 \dots dy_{n-1} dt\\
& =  \int_0^T \int_{\Omega'_{n-1}} \inside_side_term dy_1 dy_2 \dots dy_{n-1} dt \\
& = \frac{1}{(n-1)!Area(F_0)}\int_0^T \int_{F_0} \partial_{\nu}u(\nu^T)\nabla u dS_0 dt \\
&  = \frac{1}{(n-1)!Area(F_0)}\int_0^T \int_{F_0} (\partial_{\nu}u)(\partial_{\nu} \bar{u }) dS_0 dt \\
& = \frac{1}{(n-1)!Area(F_0)}\int_0^T  |\partial_{\nu} u|^2 dS_0dt.
\end{split}
\end{equation}

By equating \eqref{con:2} and \eqref{con:3} through \eqref{con: 1}, we could get our final conclusion of the observability from one face of the original simplex $\Omega \subseteq \mathbb{R}^n$:
\begin{equation} \label{con:final}
\begin{split}
\int_0^T \int_{F_0} |\partial_{\nu}u  |^2dS_0 dt & = \frac{(n-1)!Area(F_0)}{n!Vol(\Omega)} E(0)(T +\mathcal{O}(1))\\
& = \frac{TArea(F_0)}{nVol(\Omega)}E(0)\left(1+\mathcal{O}\left( \frac{1}{T}\right)\right).
\end{split}
\end{equation}
This finishes the proof of Theorem 1.1.

\appendix
\section{change variables}
\label{appendix: change variables}
 Let $A = 
\begin{bmatrix}
a_{11} & a_{12} & \dots & a_{1n}\\
a_{21} & a_{22} & \dots & a_{2n}\\
\vdots\\
a_{n1} & a_{n2} & \dots & a_{nn}
\end{bmatrix}
$. Based on the relation, we have $v(x) = v(Ay)$ where $v$ is a smooth
function.

According to chain rule, we know that,
\begin{equation}
\begin{split}
\partial_{y_j}v(Ay) & = \partial_{y_j}v
\begin{bmatrix}
a_{11}y_1 + a_{12}y_2  + \dots + a_{1n}y_n\\
a_{21}y_1 + a_{22}y_2  + \dots + a_{2n}y_n\\
\vdots\\
a_{n1}y_1 + a_{n2}y_2  + \dots + a_{nn}y_n
\end{bmatrix}\\
& = v_{x_1}a_{1j} + v_{x_2}a_{2j} + \dots + v_{x_n}a_{nj}\\
& = \sum_{k = 1}^n v_{x_k}(Ay) a_{kj}\\
& = \nabla_x^T v \Big|_{x = Ay}\begin{bmatrix}
a_{1j}\\a_{2j}\\\vdots\\a_{nj}
\end{bmatrix} , j = 1, 2, \dots, n.\\
\end{split}
\end{equation}
Then we have
\begin{equation}
\begin{split}
\nabla_y (v(Ay)) = \begin{bmatrix}
\partial_{y_1}v(Ay)\\\partial_{y_2}v(Ay)\\ \vdots \\ \partial_{y_n}v(Ay)
\end{bmatrix} & = \begin{bmatrix}
\nabla_x^T v \Big|_{x = Ay}\begin{bmatrix}
a_{11}\\a_{21}\\\vdots\\a_{n1}
\end{bmatrix} \\ 
\nabla_x^T v \Big|_{x = Ay}\begin{bmatrix}
a_{12}\\a_{22}\\\vdots\\a_{n2}
\end{bmatrix} \\
\vdots\\
\nabla_x^T v \Big|_{x = Ay}\begin{bmatrix}
a_{1n}\\a_{2n}\\\vdots\\a_{nn}
\end{bmatrix} 
\end{bmatrix}\\
& = A^T \nabla_x v(Ay)
\end{split}
\end{equation}
Therefore, $\nabla_x = (A^{-1})^T\nabla_y = B^T\nabla_y$. 

\section{Simplex Volume }
\label{appendix:volume}
We used the fact that the volume of a $n$-dimensional standard simplex is $\frac{1}{n!}$ in the main proof. The proof by induction is presented as following:
\begin{proof}
When $n=2$, the standard simplex is spanned by two vectors $\vec{v_1} = [0,1]$ and $\vec{v_2} = [1,0]$. We have its area equals to $\frac{1}{2} = \frac{1}{2!}$. 

Given a standard simplex $S \in \mathbb{R}^{n-1}$, assume $Vol(S) = \frac{1}{(n-1)!}$. Then for the standard simplex $T \in \mathbb{R}^n$ with rectangular coordinates $[t_1, t_2, \dots, t_n]$, we have this relation satisfied:
\begin{equation}
t_1 + t_2 + \dots + t_n \leq 1
\end{equation}
Assume $t_n = k$ for some constant $0\leq k \leq 1$, then: $t_1 + t_2 + \dots + t_{n-1} \leq 1-k$. The volume of this simplex $S' \in \mathbb{R}^{n-1}$ is:
\begin{equation} \label{AP: 1}
\int_{t_1 = 0}^{1-k} \int_{t_2 = 0}^{(1-k)-t_1} \dots \int_{t_{n-1} = 0}^{(1-k)-t_1-t_2-\dots - t_{n-2}} d_{t_{n-1}} \dots d{t_2} d{t_1}
\end{equation}
To use the method of integration by substitution, for each $j$ from $1$ to $n-1$, let $s_j = \frac{t_j}{1-k}$ and therefore $(1-k)ds_j = dt_j $. Regarding the bounds of the integral, all $t_j$s can be substituted by $(1-k)s_j$. In particular, we have every $s_j$ bounded by:
\begin{equation}
\begin{split}
&\Rightarrow 0 \leq (1-k)s_j \leq (1-k) - (1-k)s_1 - (1-k)s_2 - \dots - (1-k)s_{j-1} \\
& \Rightarrow 0\leq s_j \leq 1 - s_1 - s_2 - \dots - s_{j-1}
\end{split}
\end{equation}
Because we assumed that the volume of the standard simplex $S$ is $\frac{1}{(n-1)!}$, the volume of this (n-1)-dimensional simplex $S'$ can be simplified to the form of:
\begin{equation}
\begin{split}
(1-k)^{n-1}\int_{s_1=0}^{1} \int_{s_2 = 0}^{1-s_1} \dots \int_{s_{n-1}=0}^{1-s_1-s_2-\dots-s_{n-2}} ds_{n-1} \dots ds_2 ds_1 = \frac{(1-k)^{n-1}}{(n-1)!}.
\end{split}
\end{equation}

After integrating by variable $k$ from 0 to 1, we get the volume of the standard simplex $T \in \mathbb{R}^n$: 
\begin{equation}
\int_0^1 \frac{(1-k)^{n-1}}{(n-1)!} dk= \frac{1}{n!}.
\end{equation}
\end{proof}
\section{Determinant and Volume }
\label{appendix: Determinant and Volume}
Here we use the same notation indicated in the main proof. The rectangular coordinate of the original simplex, $\vec{x}$ and that of the standard simplex $\Omega'$, $\vec{y}$ is related by:
\begin{equation}
\vec{x} = A\vec{y}
\end{equation}
where $A$ is the matrix with its columns equal to vectors $\vec{p_1},\vec{ p_2}, \dots, \vec{p_n}$.

Their derivatives satisfies:
\begin{equation} \label{V:1}
dx_1 dx_2 \dots dx_n = det(A) dy_1 dy_2 \dots dy_n
\end{equation}
where $det(A)$ denotes the determinant of the Jacobian matrix. 
Furthermore, by the conclusion of Appendix \ref{appendix:volume}, we know that the volume of the standard simplex $\Omega'$ is $\frac{1}{n!}$ and thus we have:
\begin{equation}\label{V:2}
\begin{split}
&\frac{1}{det(A)} \int_{\Omega} dx_1dx_1\dots dx_n  = \int_{\Omega'} dy_1dy_2 \dots dy_n\\
&\Rightarrow \frac{1}{n!} = \frac{1}{det(A)}Vol(\Omega)\\
&\Rightarrow det(A) = n!Vol(\Omega)\\
\end{split}
\end{equation}

\bibliographystyle{alpha}
\bibliography{2020_Lu}

\begin{thebibliography}{BLR92}

\bibitem[BLR92]{BLR}
Claude Bardos, Gilles Lebeau, and Jeffrey Rauch.
\newblock Sharp sufficient conditions for the observation, control, and
  stablization of waves from the boundary.
\newblock {\em SIAM J. Control Optim.}, 30(5):1024--1065, 09 1992.

\bibitem[Chr17]{Chr17}
Hans Christianson.
\newblock Equidistribution of neumann data mass on triangles.
\newblock {\em Proc. Amer. Math. Soc.}, 145(12):5247--5255, 2017.

\bibitem[Chr19]{Chr18}
Hans Christianson.
\newblock Equidistribution of {N}eumann data mass on simplices and a simple
  inverse problem.
\newblock {\em Math. Res. Lett.}, 26(2):421--445, 2019.

\bibitem[CS19]{CS}
Hans Christianson and Evan Stafford.
\newblock Asymptotic boundary observability for the wave equation on one side
  of a planar triangle.
\newblock {\em Ann. Henri Poincar\'{e}}, 20(9):2987--3006, 2019.

\bibitem[RT74]{RT}
Jeffrey Rauch and Michael Taylor.
\newblock Exponential decay of solutions to hyperbolic equations.
\newblock {\em Indiana Univ. Math. J.}, 24(1):79--86, 1974.

\end{thebibliography}
\end{document}